\documentclass[12pt]{article}

\usepackage{amssymb,amsmath,amsfonts,amssymb}
\usepackage{graphics,graphicx,color,verbatim}

\def\@abssec#1{\vspace{.05in}\footnotesize \parindent .2in
{\bf #1. }\ignorespaces}
\graphicspath{{/EPSF/}{Figures/}}
\DeclareGraphicsExtensions{.eps}
\newtheorem{theorem}{Theorem}[section]
\newtheorem{lemma}[theorem]{Lemma}

\newtheorem{corollary}[theorem]{Corollary}

\newtheorem{remark}[theorem]{Remark}
\newtheorem{hypothesis}[theorem]{Hypothesis}
\newtheorem{property}[theorem]{Property}
\def \Rm {\mathbb R}
\def \Nm {\mathbb N}
\def \Cm {\mathbb C}

\newcommand{\eps}{\varepsilon}

\newcommand{\dsum}{\displaystyle\sum}
\newcommand{\dint}{\displaystyle\int}

\newcommand{\Mf}{\mathfrak M}
\newcommand{\Xf}{\mathfrak X} 
\newcommand{\Yf}{\mathfrak Y}

\newcommand{\Sf}{\mathfrak S}

\newcommand{\cf}{\mathfrak c}

\newcommand{\mk}{\mathfrak k}

\newcommand{\mf}{{\mathfrak f}}
\newcommand{\ff}{{\mathfrak f}}

\newcommand{\fp}{{\mathfrak p}}

\newcommand{\fu}{{\mathfrak u}}

\newcommand{\cout}[1]{}

\newcommand{\trho}{\tilde\rho}

\newcommand{\dX}{\partial X}

 \renewcommand{\arraystretch}{1.5}

\title{Reconstruction of coefficients in scalar second-order elliptic equations from knowledge of their solutions}

\author{Guillaume Bal \thanks{Department of Applied Physics and 
        Applied Mathematics, Columbia University, 
        New York NY, 10027; gb2030@columbia.edu} \and Gunther Uhlmann \thanks{Department of Mathematics, University of Washington, Seattle, WA, 98195 and University of California, Irvine, CA, 92697; gunther@math.washington.edu}}

\begin{document}
 
\maketitle

%\tableofcontents

\begin{abstract}
  This paper concerns the reconstruction of possibly complex-valued coefficients in a second-order scalar elliptic equation posed on a bounded domain from knowledge of several solutions of that equation. We show that for a sufficiently large number of solutions and for an open set of corresponding boundary conditions, all coefficients can be uniquely and stably reconstructed up to a well characterized gauge transformation. We also show that in some specific situations, a minimum number of such available solutions equal to $I_n=\frac12n(n+3)$ is sufficient to uniquely and globally reconstruct the unknown coefficients. This theory finds applications in several coupled-physics medical imaging modalities including photo-acoustic tomography, transient elastography, and magnetic resonance elastography.
\end{abstract}
 
%\begin{AMS}
%\end{AMS}

\renewcommand{\thefootnote}{\fnsymbol{footnote}}
\renewcommand{\thefootnote}{\arabic{footnote}}

\renewcommand{\arraystretch}{1.1}

%\begin{keywords}
%\end{keywords}

%\begin{AMS}
%\end{AMS}

%\pagestyle{myheadings}
%\thispagestyle{plain}

%%%%%%%%%%%%%%%%%%%%%%
%%% BEGINNING TEXT %%%
%%%%%%%%%%%%%%%%%%%%%%

\section{Introduction}
\label{sec:intro}

We consider the general second-order elliptic equation:
\begin{equation}
\label{eq:galelliptic}
\nabla\cdot a\nabla u + b\cdot \nabla u + cu =0 \quad \mbox{ in } X ,\qquad u=f \quad\mbox{ on }\dX,
\end{equation}
with complex-valued coefficients such that $a$ is a symmetric, possibly complex-valued, tensor verifying the ellipticity condition for $\alpha_0>0$:
\begin{equation}
  \label{eq:ella}
  \alpha_0 |\xi|^2 \leq \xi\cdot \Re a \xi \leq \alpha_0^{-1} |\xi|^2
\end{equation}
and with $c$ such that the above equation admits a unique solution in $H^1(X)$ for $f\in H^{\frac12}(\partial X)$. Here $X$ is an open bounded domain in $\Rm^n$ with smooth boundary $\partial X$. We assume throughout this paper that $a,b,c$ and $\nabla\cdot a$ are of class $C^{0,\alpha}(\bar X)$ for some $\alpha>0$. Elliptic regularity results \cite[Theorem 6.3.7]{M-SP-66} then ensure that the solution $u$ is a strong solution of class $C^{2,\alpha}(\bar X)$ when the boundary condition is of class $C^{2,\alpha}(\partial X)$.

We assume that we have access to internal functionals given by the complex valued solutions $u_j(x)$ of \eqref{eq:galelliptic} for a given set of boundary conditions $f=f_j$ for $1\leq j\leq I$. The main objective of this paper is to show that when $I$ is sufficiently large, then the coefficients $(a,b,c)$ can be uniquely and stably reconstructed from knowledge of the corresponding $(u_j)_{1\leq j\leq I}$ up to a natural gauge transformation.

In some specific situations, for instance when $(a,b,c)$ is close to $(a_0,0,0)$ for $a_0$ a constant complex-valued matrix satisfying \eqref{eq:ella}, then the reconstructions can be achieved for $I=I_n=\frac12n(n+3)$ coefficients, which corresponds to the (complex) dimension of the unknown coefficients $(a,b,c)$ up to the (one parameter) gauge transformation. 

The mathematical description of the measurement operator considered here and the main results of the paper are presented in section \ref{sec:main}. The proofs of the main results are detailed in sections \ref{sec:local} and \ref{sec:global}. 

The theory presented in this paper finds some applications in several recent coupled-physics (also called hybrid or multi-wave) imaging modalities that are extensively studied in the bio-engineering community.  These methods arose as an attempt to devise new imaging modalities that combine high resolution with high contrast. In section \ref{sec:appli},  we consider three such modalities: Photo-Acoustic Tomography (PAT), Transient Elastography (TE) and Magnetic Resonance Elastography (MRE).  In PAT and TE, the high resolution mechanism is ultrasound. In MRE, it is Magnetic Resonance Imaging.  Our results show that in these three imaging modalities,
all the medium parameters (some of which displaying the high contrast we are after) can be reconstructed uniquely and stably. The reason is that these coefficients have a structure that allows us to uniquely reconstruct the gauge as well. 
In PAT, the main novelty of our result is that we can reconstruct general  anisotropic
diffusion coefficients. In a scalar model for TE and MRE, the main novelty is that we can not only reconstruct anisotropic coefficients but also complex-valued coefficients that account for possible dispersion and attenuation effects.

%Several works have been devoted to the reconstruction of coefficients in elliptic equations from knowledge of the solution $u$ or that of related internal functionals.  An important application has been to coupled-physics (also called hybrid or multi-wave) imaging modalities that have received a lot of attention in recent years.  These methods arose as an attempt to combine high resolution with high contrast.  In section 5 we consider three such modalities: Photo-Acoustic Tomography (PAT), Transient Elastography (TE) and Magnetic Resonance Elastography (MRE).  In PAT and TE the high resolution modality is ultrasound. In MRE is Magnetic Resonance Imaging. 

%PAT consists of sending relatively harmless optical radiation into tissues that causes heating (with increases of the temperature in the milli Kelvin range) which results in the generation of propagating ultrasound waves (the photo-acousticeffect). Such ultrasonic waves are readily measurable.  The inverse problem then consists of reconstructing the optical properties of the tissue. Transient Elastography (TE) images the propagation of shear waves using ultrasound while Magnetic Resonance Elastography (MRE) images the same shear waves using Magnetic Resonance Imaging. PAT, TE  and MRE offer potential breakthroughs in clinical applications to early detection of cancer, functional imaging, and molecular imaging.

For the application to quantitative PAT, we refer the reader to e.g. \cite{BR-OL-11,BR-IP-11,BU-IP-10} and the references there. For applications in TE and MRE, we refer the reader to e.g. to \cite{MZM-IP-10} and its list of references.  For general references to the theory of hybrid inverse problems devoted to the mathematical analysis of similar coupled-physics imaging modalities, we refer the reader to e.g. \cite{A-Sp-08,B-IO-12,S-Sp-2011}. 

General real-valued anisotropic tensors were reconstructed in the context of ultrasound modulation in dimension $n=2$ in \cite{MB-IP-12}. Note that the reconstruction of anisotropic coefficients $a$ from boundary value measurements of $u$ (as in, e.g., the Calder\'on problem) can be performed only up to a very large class of changes of variables \cite{U-IP-09}. Moreover, the corresponding stability estimates are of logarithmic type, which corresponds to potentially drastic amplifications of measurement noise during the reconstruction. The results of this paper show that the availability of internal functionals such as those in PAT, TE, and MRE, allows one to reconstruct a larger class of coefficients and with significantly better stability estimates.

%%%%%%%%%%%%%%
\section{Main results}
\label{sec:main}

\paragraph{Gauge transform.} The elliptic equation \eqref{eq:galelliptic} may be recast as
\begin{equation}
\label{eq:galelliptic2}
a:\nabla^{\otimes 2} u + (\nabla\cdot a + b) \cdot \nabla u + cu =0.
\end{equation}
Upon multiplying through by the scalar function $\tau\not=0$, we get
\begin{equation}
  \label{eq:elltau}
  a_\tau:\nabla^{\otimes 2} u + (\nabla\cdot a_\tau + b_\tau) \cdot \nabla u + c_\tau u =0,
\end{equation}
with
\begin{equation}
\label{eq:taucoeff}
a_\tau = \tau a, \quad b_\tau = \tau b-a\nabla \tau,\quad c_\tau = \tau c.
\end{equation}
This shows that the coefficients $(a,b,c)$ can be reconstructed at most up to the above gauge transformation. We say that $(a,b,c)\sim(a_\tau,b_\tau,c_\tau)$ belong to the same class of equivalence if there exists a non vanishing (sufficiently smooth with smoothness depending on context) function $\tau$ such that \eqref{eq:taucoeff} holds. Then we say that $(a,b,c)\in\cf$ the class of equivalence.

%%%%%%%%%%%%%%%%
\paragraph{Measurement operator.}

For $f\in H^{\frac12}(\partial X)$, we obtain a solution $u\in H^1(X)$ and we can define the solution operator
\begin{equation}
\label{eq:soloperator}
\Sf_\cf: \begin{array}{rcl} H^{\frac12}(\partial X) &\to& H^1(X)\\ f &\mapsto& u = \Sf_\cf f.\end{array}
\end{equation}
Note that the solution operator is independent of the element $(a,b,c)\in\cf$. 

The main measurement operator we consider in this paper is defined as follows. Let $I\in\Nm^*$ and $f_i\in H^{\frac12}(\partial X)$ for $1\leq i\leq I$ be a given set of $I$ boundary conditions. Define $\mf=(f_1,\ldots,f_I)$. The measurement operator $\Mf_\mf$
\begin{equation}
\label{eq:finitemeasoperator}
\Mf_\mf: \begin{array}{rcl} \Xf &\to&\Yf^I\\ \cf &\mapsto&  \Mf_\mf (\cf) =(\Sf_\cf f_1,\ldots,\Sf_\cf f_I). \end{array}
\end{equation}
Here, $\Xf$ is a subset of a Banach space in which the unknown diffusion tensor is defined. That space will depend on the context. Also $\Yf$ is a subset of $H^1(X)$ where the solutions to \eqref{eq:galelliptic} are defined. The main objective of this paper is to consider settings in which $\Mf_\mf(\cf)$ for an appropriate choice of $\mf$ uniquely and stably determines $\cf$.

%%%%%%%%%%%%
\paragraph{Main results.} The main result of this paper are summarized in the following Theorem:
%%%%%%%%%%%%%%
%  MAIN THEOREM            %
%%%%%%%%%%%%%%
\begin{theorem}
\label{thm:main}
Let $\cf$ and $\tilde\cf$ be two classes of coefficients with elements $(a,b,c)$ of class $C^{m,\alpha}(\bar X)$ for $\alpha>0$ and $\nabla\cdot a$ also of class $C^{m,\alpha}(\bar X)$ for $m=0$ or $m=1$. We assume that \eqref{eq:galelliptic} is well posed for the coefficients $\cf$.

Then for $I$ sufficiently large and for an open set (for instance in the topology of $C^{2,\alpha}(\bar X)$) of boundary conditions $\ff=(f_j)_{1\leq j\leq I}$, then $\Mf_\mf(\cf)$ uniquely determines $\cf$. Moreover, for any other $\tilde\cf$ as given above, we have the stability results
\begin{equation}
\label{eq:stabmain}
\begin{array}{rcl}
 \|(a,b+\nabla\cdot a,c)-(\tilde a,\tilde b+\nabla\cdot\tilde a,\tilde c)\|_{W^{m,\infty}(X)} &\leq& C \|\Mf_{\ff}(\cf)-\Mf_{\ff}(\tilde\cf)\|_{W^{m+2,\infty}(X)}, \\
  \|b-\tilde b\|_{L^\infty(X)} &\leq& C \|\Mf_{\ff}(\cf)-\Mf_{\ff}(\tilde\cf)\|_{W^{3,\infty}(X)},
\end{array}
\end{equation}
for $m=0,1$. 

We say that $\cf$ is in the vicinity of $\tilde\cf$ if $(a,b,c,\nabla\cdot a)$ is in the $C^{0,\alpha}(\bar X)$ vicinity of $(\tilde a,\tilde b,\tilde c,\nabla\cdot \tilde a)$ for some elements $(a,b,c)\in\cf$ and $(\tilde a,\tilde b,\tilde c)\in\tilde \cf$.

Let us assume that  $\cf$ in the vicinity of either: \\[2mm] (i) $(a_0,0,0)$ for some constant diffusion tensor $a_0$; \\[2mm] (ii) $(\gamma(x) Id_n,0,c(x))$ for some scalar coefficients $\gamma\in H^{\frac n2+4+\epsilon}(\bar X)$ and $c\in H^{\frac n2+2+\epsilon}(\bar X)$ for $\epsilon>0$; \\[2mm](iii) $(\gamma(x),0,c(x))$ for an arbitrary diffusion tensor of class $H^{5+\epsilon}(X)$ and $c$ of class $H^{3+\epsilon}(X)$ in dimension $n=2$.\\[0mm]

Then  for $I=I_n=\frac12n(n+3)$ and an open set of boundary conditions $\ff$, we have that $\Mf_\mf(\cf)$ uniquely determines $\cf$. Moreover, \eqref{eq:stabmain} holds. 
\end{theorem}
The proof of the different results described in the above theorem is presented in detail in sections \ref{sec:local} and \ref{sec:global}. 

\paragraph{Reconstruction of the gauge.} In some situations the gauge in $\cf$ can be uniquely and stably determined. Let us for instance consider the specific, practically important, case of coefficients $(a,b,c)=(\gamma,0,c)$. Then we have the following result:
\begin{corollary}\label{cor:gauge}
Under the hypotheses of the preceding theorem, and in the setting where $b=0$, we have that $\Mf_\mf(\cf)$ uniquely determines $(\gamma,0,c)$. Let us define $\gamma=\tau M^0$ where $M^0$ has a determinant equal to $1$. Then we have the following stability result
\begin{equation}
  \label{eq:stabtauM0}
  \|\tau-\tilde \tau\|_{W^{1,\infty}(X)} + \| (M^0,c)-(\tilde M^0,\tilde c)\|_{L^\infty(X)}\leq  C \|\Mf_{\ff}(\cf)-\Mf_{\ff}(\tilde\cf)\|_{W^{2,\infty}(X)}.
\end{equation}
\end{corollary}
The proof of the corollary from the results stated in Theorem \ref{thm:main} may be found in section \ref{sec:recgauge}.
%%%%%%%%%%%%%%
\section{Local reconstruction}
\label{sec:local}

\subsection{Reconstruction of equivalence classes}
\label{sec:equivclass}

Let $n$ be the spatial dimension and define $I_n=\frac12 n(n+3)$. We wish to reconstruct an element in the class $\cf=(a_\tau,b_\tau,c_\tau)$ for $\tau$ an arbitrary non-vanishing function from knowledge of $u_i$ for $1\leq i\leq I_n$. We perform the reconstruction locally in the vicinity of a point $x_0$. We assume that we have constructed a solution $u_1$ such that w.l.o.g. $u_1(x_0)=1$ and by continuity $u_1\not=0$ in the vicinity of $x_0$. We then define
\begin{displaymath}
  v_j= \dfrac{u_{j+1}}{u_1}, \quad 1\leq j\leq I_n-1,\qquad \alpha = u_1^2 a, \qquad \beta = \nabla\cdot (a u_1^2) + u_1^2 b 
\end{displaymath} 
and find after some algebra that 
\begin{equation}
\label{eq:v}
 \alpha : \nabla^{\otimes2} v_j + \beta\cdot\nabla v_j=0,
\end{equation}
in the vicinity of $x_0$.

We assume that $(\nabla v_1,\ldots, \nabla v_n)$ forms a frame in the vicinity of $x_0$. Then $H=(H_{ij})_{1\leq i,j\leq n}$ with $H_{ij}=\nabla v_i\cdot\nabla v_j$ is a boundedly invertible symmetric matrix with $H^{ij}$ the coefficients of $H^{-1}$. Let us define $M_n := \frac12 n(n+1)-1$. Note that $I_n=1+n+M_n$.  We then write for $1\leq m\leq M_n$:
\begin{equation}
  \label{eq:Thetajm}
  -\nabla v_{m+n} = \Theta^m_j \nabla v_j ,\qquad \Theta^m_j = - H^{jk} \nabla v_{m+n} \cdot\nabla v_k.
\end{equation}
Here, we use the convention of summation over repeated indices. We then define
\begin{equation}
\label{eq:thetajm}
 \theta_j^m = \left\{ \begin{array}{ll}
  \Theta^m_j & 1\leq j\leq n \\ 1 & j=m+n \\ 0 & \mbox{ otherwise}
\end{array}\right. , \qquad 1\leq j\leq I_n-1, \quad 1\leq m\leq M_n .
\end{equation}
Note that we have constructed the complex-valued coefficients $\theta_j^m$ in such a way that
\begin{equation}\label{eq:vanishlincomb}
\dsum_{j=1}^{I_n-1} \theta^m_j \nabla v_j =0,\qquad 1\leq m\leq M_n.
\end{equation}
We next construct the symmetric matrices 
\begin{equation}
\label{eq:Mm}
M^m = \dsum_{j=1}^{I_n-1} \theta_j^m \nabla^{\otimes 2}v_j, \quad 1\leq m\leq  M_n.
\end{equation}
We assume that the matrices $M^m$ are linearly independent. Since the (complex) dimension of symmetric matrices equals $M_n+1$, this implies the existence of a unique symmetric, complex-valued, matrix $M^0$ such that 
\begin{equation}
  \label{eq:M0}
  M^{0} : M^m = {\rm Tr} (M^0 M^m) = \delta_{0m}, \qquad 0\leq m\leq M_n.
\end{equation}
In other words, there exists a unique normalized matrix $(M^0)^*$ that is orthogonal to the constructed $M^m$ matrices for the inner product $(A,B)={\rm Tr}(A^*B)$. The construction of $M^0$ can be obtained as follows. In the constructions presented later in the paper, the identity matrix $Id_n$ is not in the span of the matrices $M^m$. We can then use the Gram Schmidt procedure to orthonogonalize $(Id_n,(M^m)_{1\leq m\leq M_n})$ and this way construct the matrix $M^0$.

Multiplying \eqref{eq:v} by $\theta^m_j$ and summing over $j$ yields the $M_n$ constraints
\begin{displaymath}
 \alpha : M^m =0, \qquad 1\leq m\leq M_n.
\end{displaymath}
This shows that $\alpha$ is proportional to the now known matrix $M^0(x)$. 
Note that since the real part of $a$ is positive definite, the matrix $\alpha$ cannot be orthogonal to $Id_n$ for the inner product $(\cdot,\cdot)$. This justifies the fact that $(Id_n,(M^m)_{1\leq m\leq M_n})$ form a free family. 
Let us define 
\begin{equation}
\label{eq:alpha}
\alpha(x) := M^0(x).
\end{equation}
Since the matrix $\alpha$ can be reconstructed up to the gauge transformation corresponding to the multiplication by an arbitrary complex-valued function, we choose the gauge as prescribed above.

The vector field $\beta$ is then uniquely determined using \eqref{eq:v} by the explicit formula:
\begin{equation}
  \label{eq:beta}
  \beta = - H^{ij} (M^0 : \nabla^{\otimes 2} v_j ) \nabla v_i.
\end{equation}

Once $(\alpha,\beta)$ are known (up to the gauge transformation), we define
\begin{equation}
\label{eq:abc}
 a=\dfrac{1}{u_1^2} \alpha,\quad b = \dfrac{1}{u_1^2} (\beta-\nabla \cdot\alpha),\quad
 c = -\dfrac{1}{u_1} (\nabla\cdot a \nabla u_1 + b\cdot\nabla u_1).
\end{equation}
This provides an explicit reconstruction of $(a,b,c)\in \cf$, the class of equivalence, which is therefore now known. The above results may be recast as follows
\begin{equation}
\label{eq:abc2}
  a=\dfrac{1}{u_1^2} \alpha, \quad \nabla\cdot a + b = \dfrac{1}{u_1^2}(\beta-a\nabla u_1^2), \quad c=-\frac{1}{u_1} \Big( (\nabla\cdot a+b)\cdot \nabla u_1 + a: \nabla^{\otimes 2} u_1\Big).
\end{equation}

We summarize the above results in the following lemma.
\begin{lemma}\label{lem:localclass}
   Let $u_i$ for $1\leq i \leq I_n$ be solutions of the elliptic equation with boundary conditions $u_i=f_i$ on $\dX$. Let us define $v_i=u_1^{-1}u_{i+1}$ for $1\leq i\leq I_n-1$ and assume that:\\[2mm] (i) in the vicinity $X_0$ of a point $x_0$, we have that $u_1\not=0$ on $X$ and $u_1(x_0)=1$,\\[2mm] (ii) the vectors $(\nabla v_1,\ldots,\nabla v_n)$ form a frame on $X_0$ with a matrix $H_{ij}=\nabla v_i\cdot \nabla v_j$ with uniform bounded inverse on $\overline{X_0}$; \\[2mm](iii) the matrices $M^m$ for $1\leq m\leq M_n$ constructed above in \eqref{eq:Thetajm}-\eqref{eq:Mm} are linearly independent.
   
   Then the class of gauge equivalence $\cf$ is uniquely determined by $\Mf_{\ff}(\cf) = (u_i)_{1\leq i\leq I_n}$ on $X_0$ in the sense that $(u_i)_i=(\tilde u_i)_i$ implies that $\cf=\tilde\cf$ where $\tilde u_i$ are the internal functionals obtained by replacing one element in $\cf$ by one element in $\tilde\cf$. 
   
Moreover the reconstruction is stable in the sense that 
\begin{equation}
\label{eq:stabgauge}
\begin{array}{rcl}
  \|(a,b+\nabla\cdot a,c)-(\tilde a,\tilde b+\nabla\cdot\tilde a,\tilde c)\|_{L^{\infty}(X_0)} %\\
  %+\quad  \|(b,c)-(\tilde b,\tilde c)\|_{L^\infty(X_0)} 
  &\leq& C \|\Mf_{\ff}(\cf) -\Mf_{\ff}(\tilde \cf)  \|_{W^{2,\infty}(X_0)},
  \end{array}
\end{equation}
for some $(a,b,c)\in\cf$ and $(\tilde a,\tilde b,\tilde c)\in \tilde\cf$. 
\end{lemma}
The stability result is clear by inspection of the reconstruction procedure: two derivatives on $v$ are taken in the reconstruction of the matrices $M^m$ and hence of $M^0$ for instance by the Gram-Schmidt procedure, a multi-linear operation that preserves errors in the uniform norm. The same loss of derivatives is observed in the reconstruction of $(\alpha,\beta)$, and hence in $a$, $\nabla\cdot a+b$ and $c$ as can be seen in \eqref{eq:abc2}. Note that we similarly reconstruct the above coefficients in the $W^{1,\infty}$ sense, and hence $b$ in the $L^\infty$ sense, when errors are bounded in the $W^{3,\infty}$ sense as indicated in Theorem \ref{thm:main}.

\begin{remark}\label{rem:realcomplex}
  Note that the procedure described in Lemma \ref{lem:localclass} allows one to reconstruct $I_n=\frac12 n(n+3)$ {\em complex-valued} coefficients from $I_n$ {\em complex-valued} internal functionals of the form $u_j$ above or to reconstruct $I_n$ {\em real-valued} coefficients from $I_n$ {\em real-valued} functionals of the same form. 
\end{remark}

%%%%%%%%%%%%%%%%%%%%
\subsection{Reconstruction of the gauge}
\label{sec:recgauge}

Let us assume that we can reconstruct $(a_0,b_0,c_0)\in \cf$ on a domain $X$, which may be a part of the domain such as $X_0$ above or the whole domain if local reconstructions of $\cf$ are patched together to provide a global reconstruction. Let us then define
\begin{displaymath}
   (a_\tau,b_\tau,c_\tau) = (\tau a_0, \tau b_0-a_0\nabla \tau, \tau c_0) \in \cf
\end{displaymath}
an other element of the equivalence class. In this section, we show that prior information about the unknown coefficient allows us to uniquely determine the unknown gauge $\tau$. 

%Let us define $b^\flat = b_i dx^i$ the one form with coefficients given by the components of the {\em true} vector field $b=\tau b_0$. Let us assume that $b^\flat = d F$ for some scalar function $F$ so that $db^\flat=0$.

Note that another way to represent the gauge transform is to realize that
\begin{displaymath}
  (a_\tau,a_\tau^{-1}b_\tau,c_\tau) = (\tau a_0,a_0^{-1}b_0-\nabla\ln\tau,\tau c_0). 
\end{displaymath}
In other words, $a_\tau^{-1}b_\tau-a_0^{-1}b_0$ is independent of the element in the class $\cf$. If $a_\tau^{-1}b_\tau$ is seen as the 1-form $(a_\tau^{-1})_{ij} b_j dx_i$, then $d(a_\tau^{-1}b_\tau)$ is independent of the element in $\cf$ and characterizes the class of equivalence on a simply connected domain $X$. 

Let us assume that $\nabla\cdot (a_\tau^{-1}b_\tau)=\Phi$ is known. Then we observe that 
\begin{displaymath}
  -\Delta \ln\tau = \Phi - \nabla\cdot (a_0^{-1}b_0),
\end{displaymath}
so that $\tau$ is uniquely defined on a simply connected domain $X$ if it is known at the boundary $\partial X$.

As another practical assumption to reconstruct the gauge, let us assume that $b=b_\tau$ is divergence free so that $\nabla\cdot (\tau b_0 - a_0\nabla \tau)=0$, or in other words
\begin{displaymath}
  -\nabla \cdot ( a_0 \nabla \tau) + \nabla \cdot (\tau b_0) =0.
\end{displaymath} 
Note that $a_0=u_1^{-2}M^0$. This is an elliptic equation. Provided that all coefficients are real-valued and that $a_0$ is uniformly elliptic, then this equation admits a unique solution for $\tau$ when $\tau$ is known on $\dX$.  This is a consequence of the maximum principle \cite{gt1} that does not apply in the case of complex-valued coefficients. 

Let us assume the stronger constraint that $b=b_\tau=0$. This corresponds to $\Phi=0$ above. In fact, this provides the redundant system of transport equations for $\tau$:
\begin{displaymath}
  \nabla \tau = a_0^{-1} b_0 \tau, \quad \mbox{ or equivalently} \quad \nabla \ln\tau = a_0^{-1}b_0,
\end{displaymath}
which admits a unique solution provided that $\tau$ is known at one point (and admits a solution provided that $d(a_0^{-1}b_0)=0$ for $a_0^{-1}b_0$ seen as a 1-form). This reconstruction applies for arbitrary complex valued coefficients $(a,b,c)$. In the PAT, TE, and MRE applications considered in section \ref{sec:appli}, the natural setting is with $b=0$ so that the gauge can indeed be reconstructed.

Consider the specific example of 
\begin{displaymath}
\nabla\cdot \gamma \nabla u =0  \quad \mbox{ in } X, \qquad u=f \mbox{ on } \dX.
\end{displaymath}
Let us define $\gamma = \tau M^0$. The reconstruction of $M^0$ requires taking two derivatives of the data. The above equation for $\tau$ is in fact a redundant system of first-order equations for $\ln\tau$ in which we gain back one derivative. We thus obtain the unique reconstruction of $\gamma$ with the stability estimate
\begin{displaymath}
 \|\tau -\tilde\tau \|_{W^{1,\infty}(X)} + \|M^0-\tilde M^0\|_{L^\infty(X)} \leq C \|\Mf_{\ff}(\cf) -\Mf_{\ff}(\tilde \cf)  \|_{W^{2,\infty}(X)}.
\end{displaymath}
This proves Corollary \ref{cor:gauge}.
The reconstruction of the anisotropy is less stable than that of the isotropic component. 
This is consistent with similar results obtained for the ultrasound modulation problem; see \cite{MB-IP-12}.

%%%%%%%%%%%%%%%%%%%%
\subsection{Variations of the coefficients}
\label{sec:varcoef}

The above reconstruction shows that reconstructions are stable with respect to fluctuations in the measurements $\Mf(\cf)$ when the hypotheses of Lemma \ref{lem:localclass} are satisfied. We show that such hypotheses are stable with respect to small changes in the parameters $\cf$. 
\begin{lemma}
  \label{lem:regul}
  Let $u$ and $\tilde u$ be solutions of 
  \begin{displaymath}
    \nabla\cdot a \nabla u + b\cdot\nabla u + cu = \nabla\cdot\tilde a \nabla \tilde u + \tilde b\cdot\nabla\tilde u + \tilde c \tilde u=0 \mbox{ on } X_0,
\end{displaymath}
with Dirichlet conditions $u=\tilde u=f$ for $f$ of class $C^{m,\alpha}(\partial X_0)$. Then
\begin{equation}
\label{eq:regulu}
  \|u-\tilde u\|_{C^{m+2,\alpha}(X_0)} \leq C \| (\cf,\nabla\cdot a) -(\tilde \cf,\nabla\cdot\tilde a)\|_{C^{m,\alpha}(X_0)},
\end{equation}
for some positive constant $C$ independent of $\cf$ and of $\tilde\cf$ for $\tilde\cf$ bounded by $M$ in $C^{m,\alpha}$. 
\end{lemma}
\begin{proof}
  Let $w=\tilde u-u$. We find
   \begin{displaymath}
\nabla\cdot (a-\tilde a)\nabla u + (b-\tilde b)\cdot\nabla u + (c-\tilde c) u =
\nabla\cdot\tilde a \nabla w+ \tilde b \cdot\nabla w + \tilde c w.
\end{displaymath} 
The proof then follows from standard regularity results; see \cite{gt1} for the case of real-valued coefficients and \cite[Theorem 6.3.7]{M-SP-66} for the case of complex-valued coefficients.
\end{proof}
With $m=0$, we deduce that when the hypotheses of Lemma \ref{lem:localclass} are verified for the coefficients $\cf$, then they are verified with the same boundary conditions for all coefficients $\tilde\cf$ that are sufficiently close to $\cf$ in the sense given above.

%%%%%%%%%%%%%%
\section{Global reconstruction}
\label{sec:global}

We now consider several settings in which global reconstructions of $\cf$ are possible. We refer to the preceding section and the section on applications for reconstructions of the gauge under additional information.

\subsection{Global reconstructions close to constant tensor}
\label{sec:constant}

%\paragraph{Reconstruction in the vicinity of $(a,b,c)=(Id_n,0,0)$.}
We start with reconstructions in the vicinity of $a=Id_n$, $b=0$, and $c=0$, where $Id_n$ is the identity matrix in dimension $n$. The main interest of this result is that the boundary conditions $f_i$ are explicit and particularly simple. Moreover, in the case where the coefficients $(a,b,c)$ are complex-valued, the following result shows that the $I_n=\frac12 n(n+3)$ unknown complex-valued coefficients may uniquely and stably be reconstructed from exactly $I_n$ complex-valued internal functionals. The same proof shows that when all coefficients are real-valued, then the $I_n=\frac12 n(n+3)$ unknown real-valued coefficients may uniquely and stably be reconstructed from exactly $I_n$ real-valued internal functionals; see remark \ref{rem:realcomplex}.

\begin{theorem}
  \label{thm:identity}
  Let $\Xf$ be the space of $\cf$ sufficiently close to $(Id_n,0,0)$ in the sense of Lemma \ref{lem:regul} with $m=0$. Let $I_n=\frac12 n(n+3)$ and $f_i$ be the traces on $\partial X$ of $I_n$ homogeneous harmonic polynomials of degree equal to $1$ or $2$ (see the proof for the description of the polynomials). Then $\Mf_\mf$ from $\Xf$ to $\Yf^{I_n}$ is injective. Moreover, the stability result \eqref{eq:stabgauge} holds for $X_0=X$. 
\end{theorem}
This is therefore a global (in space) stability result but only for classes $\cf$ that admit an element sufficiently close to $(Id_n,0,0)$. In fact, the reconstruction works for $\cf$ close to $(a_0,0,0)$ for $a_0$ arbitrary elliptic as we observe in Theorem \ref{thm:globalconst} below.

\begin{proof}
Let $u_1=1$ be the constant solution. Let then $v_j=x_j$ for $1\leq j\leq n$ be the homogeneous polynomials of degree $1$. Finally, let us denote by $v_{ij}=x_ix_j$ and $w_{i}(x)=\frac12(x_i^2-x_{i+1}^2)$ the homogeneous harmonic polynomials of degree two for $1\leq i<j\leq n$. The other such polynomials, such as for instance $\frac12(x_1^2-x_3^2)$ can be constructed by linear combination of the polynomials $w_i$. We have thus constructed $1+n+\frac12 n(n-1)+n-1=\frac12n(n+3)=I_n$ harmonic homogeneous polynomials of degree less than or equal to $2$.

We verify that 
\begin{displaymath}
 \nabla u_1 =0, \quad \nabla v_i=e_i,\quad \nabla v_{ij} = x_je_i+x_ie_j,\quad
 \nabla w_i = x_i e_i-x_{i+1}e_{i+1}.
\end{displaymath}
Moreover, 
\begin{displaymath}
\nabla^{\otimes 2} v_{ij} = e_i\otimes e_j+e_j\otimes e_i,\qquad
\nabla^{\otimes 2} w_i = e_i\otimes e_i - e_{i+1}\otimes e_{i+1}.
\end{displaymath}

Let us define $\Theta^{ij}_k$ for $1\leq i<j\leq n$ and $1\leq k\leq n$ such that 
\begin{displaymath}
  -\nabla v_{ij} = \Theta^{ij}_k \nabla v_k, \qquad \mbox{ i.e., } \qquad \Theta^{ij}_k=
  \left\{\begin{array}{ll} 
     -x_j & \mbox{ when } k=i \\ -x_i & \mbox{ when } k=j \\ 0 & \mbox{otherwise}.
  \end{array}\right.
\end{displaymath}
Then we find that the corresponding matrices $M^m$ are defined by
\begin{displaymath}
  M^{ij} = \nabla^{\otimes 2} v_{ij} = e_i\otimes e_j+e_j\otimes e_i.
\end{displaymath}
   
Let us now define $\Theta^i_k$ for $1\leq i\leq n-1$ and $1\leq k\leq n$ such that 
\begin{displaymath}
  -\nabla w_{i} = \Theta^{i}_k \nabla v_k, \qquad \mbox{ i.e., } \qquad \Theta^{i}_k=
  \left\{\begin{array}{ll} 
     -x_i & \mbox{ when } k=i \\ x_{i+1} & \mbox{ when } k=i+1 \\ 0 & \mbox{otherwise}.
  \end{array}\right.
\end{displaymath}
Then we find that the corresponding matrices $M^m$ are defined by
\begin{displaymath}
  M^{i} = \nabla^{\otimes 2} w_{i} = e_i\otimes e_i-e_{i+1}\otimes e_{i+1}.
\end{displaymath}
The matrices $M^{ij}$ for $1\leq i<j\leq n$ and $M^i$ for $1\leq i\leq n-1$ form a free family of dimension $\frac12 n(n-1) + n-1= \frac12 n(n+1) -1=M_n$ as can easily be verified. They are orthogonal to the matrix $M^0=Id_n$.

Let now $\cf$ be close to $(Id_n,0,0)$ in the sense given in Lemma \ref{lem:regul} and let the functions $u_i$ for $1\leq i\leq I_n$ be solutions of the equation \eqref{eq:galelliptic} with boundary conditions $f_i$ that are the traces of the harmonic polynomials constructed above. (This means that $f_1=1$, $f_2=x_1$, and so on on $\dX$.)

By continuity of the solution to \eqref{eq:galelliptic} stated in Lemma \ref{lem:regul}, the linear independence of the vectors $\nabla v_j$ for $1\leq j\leq n$ still holds. The linear combinations $\theta^m_{j}$ in \eqref{eq:thetajm} and the  matrices $M^m$ in \eqref{eq:Mm} constructed by continuity from the case $\cf=(Id_n,0,0)$ still satisfy \eqref{eq:vanishlincomb} and the fact that the matrices $M^m$ are linearly independent. This ensures the existence of a matrix $M^0$ close to the identity matrix such that $\alpha=\tau M^0$ for some unknown scalar quantity $\tau$. We may then apply Lemma \ref{lem:localclass}. This concludes the proof of the theorem.
\end{proof}

We presented the above result for $\gamma$ in the vicinity of $Id_n$ in order to obtain a simple proof of a construction that satisfies the hypotheses of Lemma \ref{lem:localclass} and because the construction also appears in a later section. In fact, the result may be generalized as follows.
\begin{theorem}
\label{thm:globalconst}
The results of Theorem \ref{thm:identity} hold for $\Xf$ the space of $\cf$ sufficiently close to $(a_0,0,0)$, where $a_0$ is an arbitrary constant symmetric matrix satisfying \eqref{eq:ella}.
\end{theorem}
\begin{proof}
  The proof is very similar to that of the preceding theorem and is in some sense included in the proof of Theorem \ref{thm:global} below, to which we refer for the details. The construction of $u_1=1$ and $v_j=x_j$ is the same as that of Theorem \ref{thm:identity}. The solutions $v_{n+m}$ are then constructed as 
  \begin{displaymath}
   v_{n+m} = \dfrac12 Q_{m}x \cdot x, \qquad 1\leq m\leq \frac12n(n+1)-1,
\end{displaymath}
with $Q_{m}$ forming a family of $M_n=\frac12n(n+1)-1$ linearly independent matrices that are orthogonal to $a_0^*$, or in other words, such that $a_0:Q_{ij}=0$. The linear combinations $\theta^m_j$ are then constructed as in Theorem \ref{thm:identity} with the matrices $M^m=Q_m$ since $\nabla^{\otimes 2} v_j=0$ for $1\leq j\leq n$. This allows us to verify the hypotheses of Lemma \ref{lem:localclass} globally on $X=X_0$ for boundary conditions equal to the traces of the polynomials $1$, $x_j$, $\frac12 Q_{m}x \cdot x$, and by continuity for an open set of boundary conditions and for all coefficients $\cf$ sufficiently close to $(a_0,0,0)$.
\end{proof}

\subsection{Global reconstructions close to isotropic tensor}
\label{sec:isotropic}

%\paragraph{Reconstruction in the vicinity of $(a,b,c)=(\gamma Id_n,0,c).$} 
Let us generalize the above result by assuming that $a$ is in the vicinity of $\gamma(x) Id_n$ where $\gamma$ is a scalar real-valued (hence positive) diffusion coefficient. We still assume that $b$ is in the vicinity of $0$. Also, $c$ is  an arbitrary complex-valued potential so that \eqref{eq:galelliptic} is uniquely solvable. Then we have the following result.
\begin{theorem}
  \label{thm:isotropic}
  Let $\gamma(x)\in H^{\frac n2+4+\epsilon}(X)$ and $c(x)\in H^{\frac n2+2+\epsilon}(X)$ for $\epsilon>0$ with $\Sf_\cf$ in \eqref{eq:soloperator} bounded.
  Let $\Xf$ be the space of $\cf$ sufficiently close to $(\gamma(x)Id_n,0,c(x))$ in the sense of Lemma \ref{lem:regul} with $m=0$. Let $I_n=\frac12 n(n+3)$. There there exists an open set of $(f_i)_{1\leq i\leq I_n}$ (in any topology of sufficiently smooth functions on $\dX$) such that  $\Mf_\mf$ from $\Xf$ to $\Yf^{I_n}$ is injective. Moreover, the stability result \eqref{eq:stabgauge} holds for $X_0=X$. 
\end{theorem}
\begin{proof} 
The proof is based on the construction of complex geometrical optics solutions of the form
\begin{equation}
\label{eq:cgoisot}
   u(x;\rho) = \dfrac{1}{\sqrt{\gamma(x)}} e^{\rho\cdot x} (1+\psi_\rho(x)),
\end{equation} 
with $\rho$ a complex-valued vector such that $\rho\cdot\rho=0$.
We know that for $\gamma$ and $c$ with the aforementioned regularity and for $|\rho|$ sufficiently large, then $\psi_\rho$ is of order $|\rho|^{-1}$ in $C^2(\bar X)$ \cite{BU-IP-10}.

In the construction above Lemma \ref{lem:localclass}, we need to consider derivatives of ratios of solutions. We find that
\begin{equation}
  \label{eq:gradratios}
  \dfrac{u(x;\trho)}{u(x;\rho)}\nabla\dfrac{u(x;\rho)}{u(x;\trho)} = \rho-\trho + \varphi,\quad
  \dfrac{u(x;\trho)}{u(x;\rho)}\nabla^{\otimes2}\dfrac{u(x;\rho)}{u(x;\trho)} = (\rho-\trho)^{\otimes2} + \phi,
\end{equation}
with the vector $\varphi$ bounded independent of $(\rho,\trho)$ and matrix $\phi$ of order ${\rm max}(|\rho|,|\trho|)$ uniformly in $x\in X$.

Let us define $\rho_{ij}=\mk(e_i+ie_j)$ and define $u_{ij}=u(\cdot;\rho_{ij})$ as well as $\tilde u_{ij}=u(\cdot;\rho_{ij}^*)$. Note that $\tilde u_{ij}$ is asymptotically close to $u^*_{ij}$ as $|\mk|\to\infty$ but since $c(x)$ may be complex valued, is not necessarily equal to $u^*_{ij}$. 

We also define $\trho_{ij}=\eps\rho_{ij}$ as well as  $\rho_1=\eps^2\rho_{12}$ and $u_1=u(\cdot;\rho_1)$, with $\eps^2\mk$ sufficiently large that contributions such as $\varphi$ and $\phi$ above remain negligible for the forthcoming constructions but $\eps$ sufficiently small that $\trho=\rho_1$ or $\trho=\trho_{ij}$ in \eqref{eq:gradratios} is so small that it does not modify the independence of the matrices $M^m$ constructed below.

Let us define $v_{j}$ as follows
\begin{displaymath}
   v_1=\dfrac{u(\cdot;\trho^*_{12})}{u_1},\quad v_j=\dfrac{u(\cdot;\trho_{j-1,j})}{u_1},\quad j\geq2.
\end{displaymath}
Since all solutions $v_j$ do not vanish for $\mk$ sufficiently large, it is clear that $(\nabla v_1,\ldots, \nabla v_n)$ form a basis with $H_{ij}=\nabla v_i\cdot\nabla v_j$ a matrix with a uniformly bounded inverse for $x\in X$ (with a bound that depends on $\mk$ and $\eps$). Moreover, we find that 
\begin{displaymath}
      \dfrac{1}{v_1} \nabla v_1 \sim \eps \rho^*_{12}, 
      \,\, \dfrac{1}{v_1} \nabla^{\otimes2} v_1 \sim \eps^2 (\rho^*_{12})^{\otimes2},
      \quad \dfrac{1}{v_j} \nabla v_j \sim \eps \rho_{j-1,j},
      \,\, \dfrac{1}{v_j} \nabla^{\otimes2} v_j \sim \eps^2 \rho_{j-1,j}^{\otimes2},
      \,\, j\geq2.
\end{displaymath}
Here and below, we denote by $\sim$ equalities up to terms such as $\varphi$ and $\phi$ above that are asymptotically negligible as $|\mk|\to\infty$ as well as terms that are lower order in $\eps$.

Now for $1\leq i<j\leq n$, we define
\begin{displaymath}
    v_{ij} = \dfrac{u_{ij}}{u_1},\,\mbox{ so that } \, \dfrac{1}{v_{ij}}\nabla v_{ij}  \sim \rho_{ij}
    ,\quad\dfrac{1}{v_{ij}}\nabla^{\otimes2} v_{ij}  \sim\rho_{ij}^{\otimes 2}.
\end{displaymath}
For $1\leq i\leq n-1$, we construct
\begin{displaymath}
   \tilde v_j = \dfrac{\tilde u_{j,j+1}}{u_1},\,\mbox{ so that } 
   \,\,\dfrac{1}{\tilde v_j}\nabla \tilde v_j \sim \rho^*_{j,j+1}, \quad 
   \dfrac{1}{\tilde v_{j}}\nabla^{\otimes2} \tilde v_{j}  \sim(\rho^*_{j,j+1})^{\otimes2}.
\end{displaymath} 
Each of the vectors $\nabla v_{ij}$ and $\nabla \tilde v_j$ can uniquely be written in terms of the vectors $\nabla v_j$.  Let us define $\hat \rho_1=\rho_{12}^*$ and $\hat \rho_j=\rho_{j-1,j}$. Note that $(\hat\rho_j)_{1\leq j\leq n}$ form a basis of $\Cm^n$. Let us then introduce
\begin{displaymath}
 -\rho_{ij} = \eps \tilde\Theta^{ij}_k \hat \rho_k,\qquad  -\rho^*_{j,j+1} = \eps \tilde\Theta^{j}_k \hat\rho_k.
\end{displaymath}
Here, the summation is over the index $k$. We find that all coefficients $\tilde\Theta$ are of order $\eps^{-1}$. Then we find that
\begin{displaymath}
   -\nabla v_{ij} = \Theta^{ij}_k \nabla v_k,\,\, -\nabla \tilde v_j = \Theta^{j}_k \nabla v_k\quad\mbox{ for } \quad \Theta^{ij}_k \sim \dfrac{v_{ij}}{v_k} \tilde\Theta^{ij}_k,\qquad \Theta^{j}_k \sim \dfrac{\tilde v_{j}}{v_k} \tilde\Theta^{j}_k.
\end{displaymath}
Now for these choices, we find that 
\begin{displaymath}
   M^{ij}:=\nabla^{\otimes 2} v_{ij} + \Theta^{ij}_k \nabla^{\otimes2} v_k = v_{ij}\Big(\rho_{ij}^{\otimes2}-\tilde\Theta^{ij}_k \eps^2 \hat\rho_{k}^{\otimes2}\Big)\sim v_{ij} \rho_{ij}^{\otimes2}.
\end{displaymath}
Similarly, we have
\begin{displaymath}
   M^j :=\nabla^{\otimes 2} \tilde v_{j} + \Theta^{j}_k \nabla^{\otimes2} v_k = \tilde v_{j}\Big((\rho_{j,j+1}^*)^{\otimes2}-\tilde\Theta^{j}_k \eps^2 \hat\rho_{k}^{\otimes2}\Big)\sim \tilde v_{j} (\rho_{j,j+1}^*)^{\otimes2}.
\end{displaymath}
Note that
\begin{displaymath}
   \dfrac{\rho_{ij}^{\otimes2}}{|\mk|^2} = i(e_i\otimes e_j+e_j\otimes e_i) + (e_i\otimes e_i-e_j\otimes e_j),\quad 
   \dfrac{ (\rho^*_{ij})^{\otimes2}}{|\mk|^2}= -i(e_i\otimes e_j+e_j\otimes e_i) + (e_i\otimes e_i-e_j\otimes e_j).
\end{displaymath}
Therefore the matrices $M^{ij}$ and $M^j$ constructed above are indeed linearly independent and as in the proof of Theorem \ref{thm:identity} span a subspace of the vector space of symmetric matrices of dimension $\frac12n(n-1)+n-1=\frac12 n(n+1)-1$. The above result obtained in the limit $\eps\to0$ still holds for $\eps$ sufficiently small. Moreover, once $\eps$ is fixed, several terms of the form $\phi$ and $\varphi$ above become negligible when $|\mk|$ is sufficiently large. %This concludes the proof for $\gamma$ and $c$ sufficiently smooth.
Therefore, for $|\mk|$ sufficiently large, there exists an open set of boundary conditions $(f_i)$ such that all the hypotheses of Lemma \ref{lem:localclass} are satisfied for all $x\in X$.

%IT IS NOT CLEAR THAT WE CAN GET RID OF SMOOTHNESS CONDITION. The construction was performed for a smooth coefficient $\cf$. Let us now remove the smoothness hypothesis. Let $\tilde \cf$ be a set of coefficients that are bounded as in Lemma \ref{lem:regul} with $m=0$. Then consider a smooth $\cf$ sufficiently close to $\tilde \cf$ as in Lemma \ref{lem:regul} with $m=0$ and construct the matrices $M^m$ for $\cf$ as above. Then applying Lemma \ref{lem:regul}, we find that for the same boundary conditions, the matrices $\tilde M^m$ constructed with the coefficients $\tilde \cf$ are still linearly independent for all $x\in X$. 

This proves the result for $\tilde\cf$ of the form $(\gamma(x)Id_n,0,c(x))$. Now by continuity and Lemma \ref{lem:regul}, the same boundary conditions can be used to satisfy the requirements of Lemma \ref{lem:localclass} for all $\cf$ sufficiently close to $\tilde \cf$.
\end{proof} 

\subsection{Global reconstructions in two dimensions}
\label{sec:2d}

The above reconstruction procedure has been proved to hold in the vicinity of $(Id_n,0,0)$ or $(\gamma(x)Id_n,0,c(x))$. In this section, we generalize the result to proving that global reconstructions are possible for coefficients in the vicinity of $(\gamma(x),0,c(x))$ where $\gamma$ is an arbitrary real-valued second-order elliptic tensor in dimension $n=2$ and $c(x)$ is a complex-valued potential:

\begin{theorem}
  \label{thm:twodim}
  Let $\Xf$ be the space of $\cf$ sufficiently close to $(\gamma,0,c)$ in the sense of Lemma \ref{lem:regul} with $m=0$ with $\gamma$ of class $H^{5+\epsilon}(\bar X)$ and $c$ of class $H^{3+\epsilon}(\bar X)$ for $\epsilon>0$. Let $I_2=5$. Then there exists an open set of boundary conditions $\ff=(f_i)_i$ such that $\Mf_\mf$ from $\Xf$ to $\Yf^{I_2}$ is injective. Moreover, the stability result \eqref{eq:stabgauge} holds for $X_0=X$. 
\end{theorem}
\begin{proof}
To simplify the notation, we set $c\equiv0$ and leave the details to the reader to consider the case $c\not=0$ as was done in the proof of Theorem \ref{thm:isotropic}. 

   We prove that we can apply Lemma \ref{lem:localclass} for $X_0=X$ and $(a,b,c)=(\gamma,0,0)$ for an open set of boundary conditions $\ff$. Then by continuity, the hypotheses of Lemma \ref{lem:localclass} still hold for $\cf$ sufficiently close to $(\gamma,0,0)$. As we did in the proof of the preceding theorem, we can assume that $\gamma$ is smooth since by an application of Lemma \ref{lem:regul}, the result can be extended to any $\gamma$ satisfying the regularity hypotheses of Lemma \ref{lem:regul} with $m=0$.
   
   This global reconstruction works only in two dimensions of space and for real valued tensors $\gamma$. The reason is that global complex geometrical optics solutions can be constructed in two dimensions by means of appropriate quasiconformal maps. Such results do not hold in general in dimension $n\geq3$.
   
Let $G_0$ be the identity conformal structure and $G=G(z)$ the conformal structure given by $\gamma$. Then there is a diffeomorphism $\phi$ from $\Cm\to\Cm$, unique after normalization at infinity,  such that \cite{AIM-PUP-08}
\begin{displaymath}
   D\phi^t(z) J^{-1}(z,\phi(z))  D\phi(z) = G(z).
\end{displaymath}
Then with $\phi^*u=u\circ \phi$ we have that
\begin{displaymath}
\nabla\cdot J^{-1} \nabla \phi^* u = \nabla\cdot G\nabla u \circ \phi.
\end{displaymath}
Identifying $z=x_1+ix_2$ and $x=(x_1,x_2)$ and taking $\phi^*u(x) = \sqrt J(\phi(x)) e^{\rho\cdot x}$, we find that $u(x) = \sqrt J(x) e^{\rho \cdot \varphi(x)}$ with $\varphi(x)=(\varphi_1(x),\varphi_2(x))$ a diffeomorphism with $\varphi^{-1}=\phi$. In other words, we construct CGO solutions for anisotropic media of the form
\begin{displaymath}
   u_\rho(x) = \sqrt J(x) e^{\rho\cdot \varphi(x)} (1+\psi_\rho(x)).
\end{displaymath}
Let $\rho_1=k(ie_1+e_2)$ and $\rho_2=k(ie_2-e_1)$. Let $u_j=u_{\rho_j}$, $j=1,2$. As in the preceding section, we also define $\tilde u_j=u_{\rho^*_j}$ for $j=1,2$. Then we find 
\begin{displaymath}
k^{-1}\nabla u_1 = (i\nabla \varphi_1+\nabla\varphi_2) u_1 +\zeta_1,\qquad
k^{-1}\nabla u_2 = i (i\nabla \varphi_1+\nabla\varphi_2) u_2 +\zeta_2.
\end{displaymath}
We thus find
\begin{displaymath}
   k^{-1}u_2^{-1}\nabla u_2 = i k^{-1}u_1^{-1}\nabla u_1  + o(1) = -\nabla\varphi_1+i\nabla\varphi_2 + o(1).
\end{displaymath}
Here, we are decomposing $\nabla u_2$ over $\nabla u_1$ and $\nabla \tilde u_1$, which form a basis for $k$ sufficiently large. 
Then
\begin{displaymath}\begin{array}{rcl}
   M &=&  i k^{-2}u_1^{-1}\nabla\otimes\nabla u_1 - k^{-2}u_2^{-1}\nabla\otimes\nabla u_2 \\ &\sim &(i+1) \big(\nabla\varphi_2^{\otimes 2}-\nabla\varphi_1^{\otimes 2}\big) + (i-1) \big(\nabla\varphi_1\otimes\nabla\varphi_2+\nabla\varphi_2\otimes\nabla\varphi_1\big),
   \end{array}
\end{displaymath}
in the limit $k\to\infty$. In the same way that we have decomposed $\nabla u_2$ over $\nabla u_1$ and $\nabla \tilde u_1$ above, we can decompose $\nabla \tilde u_2$  over $\nabla u_1$ and $\nabla \tilde u_1$ as well. In the limit $k\to\infty$, the matrix $M^*$ will thus be given by the complex conjugation of the above matrix. This proves that by in the limit $k\to\infty$, the matrices $M^m$ that we construct are given by the real and imaginary parts of $M$:
\begin{displaymath}
 M_{\pm} = \big(\nabla\varphi_2^{\otimes 2}-\nabla\varphi_1^{\otimes 2}\big) \pm \big(\nabla\varphi_1\otimes\nabla\varphi_2+\nabla\varphi_2\otimes\nabla\varphi_1\big).
\end{displaymath}
After change of coordinates, we obtain the two matrices:
\begin{displaymath}
  M_1 = \nabla\varphi_2^{\otimes 2}-\nabla\varphi_1^{\otimes 2},\qquad M_2 = \nabla\varphi_1\otimes\nabla\varphi_2+\nabla\varphi_2\otimes\nabla\varphi_1
\end{displaymath}
which we want to be non trivial and linearly independent. The above matrices $M_{1,2}$ are those obtained in the limit $k\to\infty$. This means that for $k$ sufficiently large, the two constructed matrices $M_{1,2}$ from $u_1$ and $u_2$ will be close to their limits and hence satisfy the same properties of linear independence. 

Now we observe that $a=\nabla \varphi_1$ and $b=\nabla\varphi_2$ are linearly independent since $\varphi$ is a diffeomorphism. And $a\otimes a$, $b\otimes b$, $a\otimes b+b\otimes a$ are basis elements for symmetric matrices. Thus $M_1$ has coordinates $(-1,1,0)$ while $M_2$ has coordinates $(0,0,1)$ in that basis. As a consequence, both matrices $M_1$ and $M_2$ are linearly independent, in the limit $k\to\infty$ as well as for $k$ sufficiently large. Note that the independence is uniform in $x\in X$ for $k$ sufficiently large.

This shows that the hypotheses of Lemma \ref{lem:localclass} are satisfied for $X_0=X$. Such a calculation holds for any set of coefficients close to $(\gamma,0,0)$. A very similar proof applies to $\cf$ in the vicinity of $(\gamma,0,c)$ as stated in the theorem. This proves the theorem.
\end{proof}

\begin{comment}
Let us conclude this section by a remark [THAT MAY DISAPPEAR] on the reconstruction in the limti $k\to\infty$.
Let us define $H_{ij}=\nabla\varphi_i\cdot\nabla\varphi_j$ and $\gamma=\gamma_{ij}\nabla\varphi_i\otimes\nabla\varphi_j$. The constraints
\begin{displaymath}
  M_1 : \gamma = M_2 :\gamma =0,
\end{displaymath}
are equivalent to 
\begin{displaymath}
  (H\gamma H)_{11} = (H\gamma H)_{22},\quad (H\gamma H)_{12}=0.
\end{displaymath}
In other words, we find that 
\begin{displaymath}
 \gamma(x)= \alpha(x) H^{-2}(x),
\end{displaymath}
for $\alpha(x)$ a positive function. This formula works for $M_{j}$ approximated as given above. When $\zeta_j$ is not neglected, then $M_j$ is slightly different. Then we need to complete $M_1$ and $M_2$ with $M_3$ as a basis of symmetric matrices. Moreover, we choose $M_3$ orthogonal (for the Frobenius norm) to $M_1$ and $M_2$. We then find that $\gamma$ is proportional to $M_3$. This is the generalization of the preceding calculation.  
\end{comment}

\subsection{Global reconstructions with redundant measurements}
\label{sec:nd}

 In this section, we show that reconstructions are possible for essentially arbitrary (sufficiently smooth) coefficients $\cf$. However, the construction of the matrices $M^m$ becomes local. We thus need to use a number of internal functionals $I$ that is potentially much larger than $I_n$, although we do not expect this large number of coefficients to be necessary in practical inversions.

The local constructions require that certain properties of linear independence be satisfied. Such conditions will be satisfied for well-chosen illuminations $f_j$ on the boundary $\partial X$. The control of the linear independence from the boundary is obtained by means of a Runge approximation; see Lemma \ref{lem:runge} below. This step requires that the operator $L=\nabla\cdot a\nabla + b\cdot\nabla+c$ satisfy a unique continuation principle, which we state as follows:
\begin{property}[Unique Continuation]\label{property:UCP}
 We say that $L$ satisfies the unique continuation principle when $Lu=0$ on $X\backslash X_0$ with $u=0$ on $\partial X$ and $n\cdot\nabla u=0$ on $\partial X$ implies that $u=0$ on $X\backslash X_0$, where $X_0$ is an arbitrary sufficiently smooth open domain $X_0\subset\subset X$. 
\end{property} 
For unique continuation results, we refer the reader to \cite{C-AJM-58,N-CPAM-57} and the theoretical results we shall use here \cite[Theorem 17.2.1]{H-III-SP-94}. The latter result states that $L$ satisfies the unique continuation principle \ref{property:UCP} when the principal symbol of $L$ given by $p(x,\xi)=a(x)\xi\cdot\xi$ is such that:
\begin{quote}
  (i) $a(x)$ is Lipschitz continuous, \\ [1mm]
   (ii) For $\xi,N \in\Rm^n\backslash \{0\}$, the quadratic equation $p(x,\xi+\tau N)=0$  in the variable $\tau\in\Cm$ admits a double root $\tau$ if and only if $\xi+\tau N=0$. 
\end{quote}
Then we have the following lemma:
\begin{lemma}\label{lem:UCP}
  Let $p(x,\xi)=a(x)\xi\cdot \xi$ be the principal symbol of $L$, which we assume is elliptic. 
  \\ In dimension $n\geq3$, the quadratic equation $p(x,\xi+\tau N)=0$ for $\xi,N \in\Rm^n\backslash \{0\}$ never admits a double root $\tau$ unless $\xi+\tau N=0$. \\
  In dimension $n=2$, the same result holds when in addition \eqref{eq:ella} is satisfied.\\
  In all these cases, $L$ thus satisfies Property \ref{property:UCP} when $a$ is Lipschitz continuous.
\end{lemma}
\begin{proof}
   The proof is essentially given in  \cite[Lemma 17.2.5]{H-III-SP-94}. In dimension $n\geq3$, the equation $p(x,\xi+\tau N)$ has one root with $\Im\tau>0$ and one root with $\Re\tau>0$. In dimension $n=2$, the equation $p(x,\xi+\tau N)$ has a double root at a fixed point $x\in X$ if and only if we have $p(x,\xi)=(l(x)\cdot\xi)^2$ for some complex-valued vector $l=l_r+il_i\in\Cm^n$ (with $l_r=\Re l$ and $l_i=\Im l$) to preserve ellipticity (note that such quadratic forms cannot be elliptic in dimension $n\geq3$). But then $p(x,\xi)=(l_r\cdot\xi)^2-(l_i\cdot\xi)^2 + i l_r\cdot \xi l_i\cdot\xi$ so that the real part of $a$ is not elliptic. This proves the lemma.
\end{proof}
\begin{comment}
This result is not known to hold for arbitrary (even sufficiently smooth) coefficients $\cf$. It is, however, known for large classes of coefficients that we now describe:
\begin{hypothesis}[H]\label{hyp:H}
 Let $(a,b,c,\nabla\cdot a)$ be coefficients of class $C^{0,\alpha}(X)$ for some $\alpha>0$.  We say that the corresponding class $\cf$ satisfies hypothesis (H)  when one of the following assumptions is satisfied:
 \\[4mm]
 (i) $a$ is real-valued; \\[3mm]
 (ii) $(a,b,c)$ are analytic coefficients on $X$; \\[3mm]
 (iii) Define $p(x,\xi)=a(x)\xi\cdot\xi$. Let $\xi,N \in\Rm^n\backslash \{0\}$. We assume that the quadratic equation $p(x,\xi+\tau N)=0$  in the variable $\tau\in\Cm$ admits a double root $\tau$ if and only if $\xi+\tau N=0$. 
\end{hypothesis}
Under any of the above hypotheses, it is known that $L$ then satisfies the unique continuation principle \ref{property:UCP} \cite[Theorem 17.2.1]{H-III-SP-94}. When the coefficients are real analytic, then unique continuation is an application of the Holmgren theorem \cite{H-I-SP-83}. For $a$ real-valued, we can easily verify that (i) is a consequence of the more general (iii). 
\end{comment}

With this result, we can now state the main theorem of the paper.
\begin{theorem}
  \label{thm:global}
  Let $\Xf$ be the space of coefficients $\cf$ such that $(b,c,\nabla\cdot a)$ are of class $C^{0,\alpha}(X)$, $a$ is of class  $C^{0,1}(X)$, and such that \eqref{eq:ella} holds. Then there exists $I\geq I_n$ and an open set (for the topology of $C^{2,\alpha}(\partial X)$) of boundary conditions $\ff=(f_i)_{1\leq i\leq I}$ such that $\Mf_\mf$ from $\Xf$ to $\Yf^{I}$ is injective. Moreover, the stability result \eqref{eq:stabgauge} holds for $X_0=X$. 
\end{theorem}
\begin{proof} We decompose the proof into three steps: we first construct local solutions assuming that the coefficients are constant. We then extend the local constructions to the case of non-constant coefficients. We finally apply the Runge approximation to obtain an open set of boundary conditions such that the hypotheses of Lemma \ref{lem:localclass} are satisfied locally. Local constructions are then patched together to provide global stable and unique reconstructions.
\paragraph{\bf Problem with constant coefficients.}
Let first $x_0$ be a point inside $X$, which by change of coordinates we call $0$. 
Let us define $(a_0,b_0,c_0)=(a(0),b(0)+\nabla\cdot a(0),c(0))$. We then look for solutions of the constant coefficient equation
\begin{equation}\label{eq:fu}
 L_0 \fu :=  a_0: \nabla^{\otimes 2} \fu + b_0\cdot\nabla \fu + c_0 \fu =0.
\end{equation}
We look for solutions approximately of the form
\begin{displaymath}
  \fp = \dfrac12 Qx\cdot x + \rho\cdot x + d, \qquad \nabla \fp = Qx+\rho,\qquad \nabla^{\otimes 2} \fp = Q.
\end{displaymath}
In order for $\fp$ to satisfy the equation at $x=0$, we need to find $(Q,\rho,d)$ such that
\begin{displaymath}
   a_0 : Q + b_0 \cdot \rho + c_0 d =0.
\end{displaymath}
We construct $I_n = 1+n+\frac12 n(n+1)-1$ such solutions below. We then realize that $L_0\fp=O(x)$. Let $r_0$ be sufficiently small and let us define
\begin{equation}
\label{eq:fux0}
   L_0 \fu =0 \quad\mbox{ in } \quad B(0,r_0), \qquad \fu=\fp\quad \mbox{ on } \partial B(0,r_0).
\end{equation}
For $r_0$ sufficiently small, the derivatives up to order two of $\fu$ and $\fp$ are very close. The linear independence of the structures constructed below with the polynomials $\fp$ at $x=0$ therefore still holds for the corresponding structures constructed with the elliptic solutions $\fu$ in \eqref{eq:fux0}.

We call the first solution $\fu_0$ obtained by defining
\begin{displaymath}
   d =1 ,\qquad \rho =0,\qquad Q = -\frac{d a_0^*}{a_0:a_0^*}.
\end{displaymath}
Note that $\fu_0$ does not vanish in a sufficiently small neighborhood of $0$ (and can be normalized so that $\fu_0(0)=1$). We next define the solutions $\fu_j$ for $1\leq j\leq n$. The vector $b_0=b_{0r}+ib_{0i}$ is after a rotation if necessary in the span of $e_1$ and $e_2$. We thus write $b_0=\mu e_1+\nu e_2$ for $\mu$ and $\nu$ in $\Cm$. For $j=1,2$, we define
\begin{displaymath}
   d_j=0,\qquad \rho_j= e_j,\qquad  Q_1 =  -\dfrac{\mu a_0^*}{a_0:a_0^*}, \qquad  Q_2 =  -\dfrac{\nu a_0^*}{a_0:a_0^*}.
\end{displaymath}
For $j\geq3$, we define
\begin{displaymath}
  d =0,\qquad \rho = e_j,\qquad Q =0.
\end{displaymath}
The solutions $\fu_j$ are therefore constructed such that $\nabla \fp_j=e_j$ at $x=0$. Moreover, we find that $\nabla \frac{\fp_j}{\fp_0}=e_j$ at $x=0$ as well since $\nabla \fp_0=0$ at $x=0$. We thus obtain that $\nabla  \frac{\fu_j}{\fu_0}$ form a basis of $\Rm^n$ in a sufficiently small neighborhood of $0$. 

Finally, for $n+1\leq j\leq n+\frac12 n(n+1)-1$, we define $d=0$ and $\rho=0$ 
%\begin{displaymath}
%   d=0,\qquad \rho=0, 
%\end{displaymath}
and choose the matrices $Q_j$ such that they form a free family of symmetric matrices that are orthogonal to $a_0^*$, the complex conjugate of $a_0$. This free family has dimension $\frac12 n(n+1)-1$. This implies that $(a_0:Q_j)=0$. 

In the construction above Lemma \ref{lem:localclass}, it is $\nabla^{\otimes 2}v_{j+m}$ for $v_{j}=\fu_0^{-1}\fu_j$ that is used  to form a free family of dimension $\frac12 n(n+1)-1$. We verify that 
\begin{displaymath}
  \nabla ^{\otimes 2} \frac{\fu_{j+m}}{\fu_0} = \dfrac{1}{\fu_0} \nabla^{\otimes 2} \fu_{j+m} + \nabla \fu_{j+m} \otimes \nabla \frac1{\fu_0} + \nabla \dfrac1{\fu_0} \otimes\nabla \fu_{j+m} + \fu_{j+m} \nabla^{\otimes 2} \frac1{\fu_0}.
\end{displaymath}
We verify that both sides equal $Q_j$ at $x=0$ when $\fu_j$ is replaced by $\fp_j$.

Let now $\theta_j^m$ and $M^m$ be defined as above Lemma \ref{lem:localclass}. We verify that $M^m$ is close to $Q_{j+m}$ at $x=0$. By continuity, the matrices $M^m$ are therefore linearly independent in a ball $X_0=B(0,r_0)$ for $r_0>0$ sufficiently small. This shows that on that ball, the family of matrices $M^m$ as constructed above Lemma \ref{lem:localclass} satisfy the hypotheses of that lemma. All other hypotheses of that Lemma are therefore satisfied for the family $\fu_j$.

\paragraph{\bf Problem with non-constant coefficients.} We now return to the full problem and look for solutions of the form:
\begin{equation}\label{eq:localu}
 a(x): \nabla ^{\otimes 2} u + (b+\nabla\cdot a) \cdot\nabla u + c u =0 \quad\mbox{ in } X_0,\qquad u = \fu \quad \mbox{ on } \partial X_0,
\end{equation}
where $X_0=B(0,r_0)$ is a ball whose radius $r_0$ is equal to or smaller than the value chosen in the construction of $\fu$. Let $w=\fu - u$. We find
\begin{displaymath}
 a_0: \nabla ^{\otimes 2} w + b_0 \cdot\nabla w + c_0 w = (a-a_0): \nabla^{\otimes 2}u + (b+\nabla\cdot a - b_0)\cdot \nabla u + (c-c_0) u
\end{displaymath}
on $X_0$ with $w=0$ on $\partial X_0$. %Using Lemma \ref{lem:regul}, we may assume that all the coefficients are of class $C^\infty$. 
By assumption on the coefficients and $u$, the above right-hand side is bounded uniformly by $r_0^\alpha$ on the ball $X_0$.  We deduce from elliptic regularity results for complex-valued coefficients \cite[Chapter 6]{M-SP-66} that 
\begin{displaymath}
  \|w\|_{C^{2,\alpha}(X_0)} \leq Cr_0^\alpha,
\end{displaymath}
for some positive constant $C$. Thus for $r_0$ sufficiently small, we find that the functions $u_j$ defined as solutions for \eqref{eq:localu} with boundary conditions $\fu_j$ are arbitrarily close to $\fu_j$ in the $C^2$ sense for $r_0$ sufficiently small. This proves that $u_0$ remains non-vanishing and close to $1$ on $X_0$, that $\nabla\frac{u_j}{u_0}$ for $1\leq j\leq n$ remain linearly independent, and that the matrices $M^m$ constructed above Lemma \ref{lem:localclass} satisfy the independence properties stated in that lemma.

\paragraph{\bf Continuation to the boundary.} 
So far, we have constructed solutions $u$ that are defined on $X_0=B(0,r_0)$. We need to construct solutions on the whole domain $X$ such that their restrictions on $X_\mu=B(0,\mu r_0)$ is a sufficiently accurate approximation of $u$ for $0<\mu<1$. We need the following Runge approximation property, following \cite{NUW-JPAM-05}; see also \cite{L-CPAM-56}.
\begin{lemma}[Runge approximation]\label{lem:runge}
   Let $L$ be an operator satisfying the unique continuation property of Cauchy data on $X$ as described above.
   
   Let $u_0$ be a solution of $Lu_0=0$ on $X_0$ and let $X_\mu=B(0,\mu r_0)$ for $0<\mu<1$. Then for each $\eps>0$, there is a function $f_\eps\in H^{\frac12}(\partial X)$ such that the solution of $Lu_\eps=0$ on $X$ with $u_\eps=f_\eps$ on $\partial X$ is such that 
\begin{equation}\label{eq:receps}
     \|u_\eps-u_0\|_{C^{2,\alpha}(X_\mu)} \leq \eps.
\end{equation}
\end{lemma}
\begin{proof}[Runge Lemma].
     Let $E=\{u\in H^1(X_0),\, Lu=0 \mbox{ in } X_0\}$ and $F=\{u_{|X_0}, \,\, u\in H^1(X),\, Lu=0 \mbox{ in } X\}$ be linear subspaces of $L^2(X_0)$. We wish to prove that $\bar F=E$ for the strong $L^2$ topology. By Hahn Banach, this means that for all $f\in L^2(X_0)$, then $(f,u)=0$ for all $u\in F$ implies that $(f,u)=0$ for all $u\in E$. 
     
Let us extend $f$ by $0$ outside $X_0$ and still call $f$ the extension on $X$. Define then
\begin{displaymath}
         L^* v =f \quad\mbox{ in } \quad X ,\qquad v=0\quad\mbox{ on } \quad \partial X.
\end{displaymath}
Here $L^*=\nabla\cdot a^*\nabla - b^*\cdot\nabla + c^*$ is the formal adjoint to $L$. Note that $v$ is well-defined since $\Sf_\cf$ in \eqref{eq:soloperator} is assumed to be bounded. Integrations by parts show that
\begin{displaymath}
   (Lu,v)-(u,L^*v) = \dint_{\partial X} (an\cdot\nabla u v^* - an \cdot\nabla v^* u - b\cdot n u v^*) d\sigma
    = \dint_{\partial X} an\cdot\nabla v^* u d\sigma.
\end{displaymath}
Since this holds for any function $u\in H^1(X)$ and hence for any $u_{|\partial X}\in H^{\frac12}(\partial X)$, we deduce that 
$ a^*n\cdot\nabla v=0$ on $\partial X$. We thus find
\begin{displaymath}
    L^*v =0 \quad\mbox{ in } \quad X\backslash X_0 ,\qquad v=0\,\mbox{ and }\, a^*n\cdot\nabla v =0 \quad\mbox{ on } \quad \partial X_0.
\end{displaymath}
We use the unique continuation assumption to deduce that $v\equiv0$ in $H^1(X\backslash X_0)$ so that $v=0$ in the $H^{\frac12}(\partial X_0)$ sense and $a^*n\cdot\nabla v=0$ in the $H^{-\frac12}(\partial X_0)$ sense. For any $u\in E$, we thus find that $(f,u)=0$, which thus proves that $\bar F=E$. This shows that  $u-u_0$ is arbitrarily small in  $L^2(X_0)$. Now regularity results as they are written for instance in \cite[Theorem 6.2.5]{M-SP-66} for elliptic problems with complex coefficients such that $(a,b,c,\nabla\cdot a)$ are of class $C^{0,\alpha}$ for $\alpha>0$ (see also \cite[Theorem 17.2.7]{H-III-SP-94}), allow us to conclude that \eqref{eq:receps} holds.  Indeed, we have an equation $L(u_0-u_\eps)=0$ on $X_0$. We then get the required interior regularity of $u_0-u_\eps$ in $C^{2,\alpha}(B(0,\mu r_0))$  for all $\mu<1$. 
\end{proof}

We now conclude the proof of Theorem \ref{thm:global}.
The  uniqueness to the Cauchy problem is guaranteed by Lemma \ref{lem:UCP}. We have obtained, using the Runge approximation, the construction of a family $u_j$ for an open set of boundary conditions $f_j$  such that the hypotheses of Lemma \ref{lem:localclass} are satisfied on $X_0$. It remains to cover $X$ by a finite number of balls of radius $\mu r_0$ (for $\mu<1$ as necessary to apply the Runge approximation result) and to apply Lemma \ref{lem:localclass} globally on $X$ and obtain a unique and stable reconstruction of $\cf$ on $X$.
\end{proof}

\section{Applications to coupled-physics inverse problems}
\label{sec:appli}

The salient feature of coupled-physics inverse problems (also known as hybrid inverse problems) is that they involve a high resolution modality and a high contrast modality to obtain a coupled (hybrid) modality imaging combining both high contrast with high resolution. We consider three such families of coupled physics inverse problems that may be modeled by the theory developed in the preceding sections; quantitative photo-acoustic tomography (QPAT), transient elastography (TE), and the mathematically similar modality called magnetic resonance elastography (MRE); see \cite{B-IO-12} for a review on hybrid inverse problems.

\subsection{Quantitative Photo-Acoustic Tomography}

The first modality we consider is called quantitative photo-acoustic tomography (QPAT). The high contrast modality is optical tomography. The ultimate objective of QPAT is the reconstruction of the optical coefficients in an elliptic equation. 

Radiation propagation is modeled by the following equation
\begin{equation}
\label{eq:ellPAT}
-\nabla \cdot\gamma \nabla u_j + \sigma u_j =0 \quad \mbox{ in } X,\qquad u_j=f_j\quad \mbox{ on } \dX.
\end{equation}
Here, $\gamma$ is the real-valued diffusion tensor and $\sigma$ the real-valued absorption coefficient.

The high resolution modality is ultrasound. A first well posed inverse wave (ultrasound) problem is solved to reconstruct internal functionals of the unknown coefficients.  This first step of QPAT provides access to the following internal functionals \cite{BR-OL-11,BR-IP-11,BU-IP-10,SU-IP-09}
\begin{equation}
\label{eq:IFQPAT}
H_j(x) = \Gamma(x) \sigma(x) u_j(x) \quad \mbox{ in } X.
\end{equation}
Here $\Gamma(x)$ is the Gr\"uneisen coefficient, which is assumed to be known in this paper and, therefore, without loss of generality assumed to equal $1$.
We assume that $1\leq j \leq I$, with $I$ the number considered in the preceding sections. We assume that all coefficients are known on $\dX$ and that $f_1>0$ on $\dX$ so that $u_1>0$ by the maximum principle. Then multiplying the above equation for $u_1$ by $u_j$ and for $u_j$ by $u_1$ and subtracting the results, we get
\begin{displaymath}
   -\nabla\cdot (\gamma u_1^2) \nabla \dfrac{u_j}{u_1} = -\nabla\cdot (\gamma u_1^2) \nabla \dfrac{H_j}{H_1}  =0.
\end{displaymath}
Therefore if the $I-1$ conditions $(f_2,\ldots, f_{I})$ boundary conditions are chosen as in the preceding section, we obtain that $\gamma u_1^2$ can be uniquely and stably reconstructed.  Indeed, we are here in the setting where $b=0$, which allows one to reconstruct the gauge and hence the whole diffusion tensor $\gamma$ as indicated in Corollary \ref{cor:gauge}.

Now the equation for $u_1$ may be recast as 
\begin{equation}
  \label{eq:u1inv}
  -\nabla\cdot(\gamma u_1^2) \nabla \dfrac{1}{u_1} = H_1 \quad\mbox{ in } X,\qquad \dfrac1{u_1}=\dfrac1{f_1}\quad\mbox{ on } \dX.
\end{equation}
This uniquely determines $u_1$ and hence $\gamma_1$ in a stable fashion. Since $H_1=\sigma u_1$, this also determines $\sigma$ uniquely and stably. This concludes the derivation of the unique and stable reconstruction of $(\gamma,\sigma)$ from QPAT measurements when the Gr\"uneisen coefficient is known.

Note that the same elliptic equation \eqref{eq:ellPAT} with $\sigma=0$ has been used to reconstruct a scalar diffusion coefficient from knowledge of $u$ by solving the transport equation \eqref{eq:ellPAT} for $\gamma$ with applications in underground water flows \cite{A-AMPA-86,R-SIAP-81}.

\subsection{Coupled-physics methods based on Elastography}

In this section, the high contrast modality is elastography; see \cite{MZM-IP-10} and reference there for more details. The elastic (stiffness) properties of tissues are to be reconstructed. We assume here that the elastic displacements are modeled by a time-harmonic scalar equation of the form
\begin{equation}
\label{eq:elast}
\nabla\cdot \gamma(x) \nabla u_j + \omega^2\rho(x) u_j =0 \quad\mbox{ in } X,\qquad u_j=f_j \quad \mbox{ on } \dX.
\end{equation}
Here, $\gamma$ is a tensor-valued, possibly complex-valued, Lam\'e parameter and $\rho$ is a density that may also be complex-valued in full generality. The reason for these coefficients to be complex-valued is that elastic waves are attenuated by various dispersion effects. In the frequency domain, such attenuation effects take the form of complex-valued coefficients. Elastographic tomography was one of the main motivations to consider the reconstruction of complex-valued coefficients in the preceding section. 

In transient elastography (TE), the high resolution modality is again ultrasound. As comparatively slow elastic waves propagate through the domain $X$ of interest, ultrasound measurements are used to infer the internal displacements, i.e., the solution of the Helmholtz equation \eqref{eq:elast}. In Magnetic resonance elastography (MRE), the high resolution modality is magnetic resonance. Along with elastic displacements, proton displacements occur that can be measured by an MRI machinery. 
  
In both modalities, the internal functionals obtained by ultrasound in TE and by magnetic resonance imaging in MRE are given by the displacements:
\begin{equation}
  \label{eq:HTE}
  H_j(x) = u_j(x) \quad \mbox{ in } X.
\end{equation}
This is exactly the setting considered in Theorem \ref{thm:main}. Note, however, that in many applications of elastography, the scalar model considered here is not sufficiently accurate. Generalizations to more precise models of linear or nonlinear elasticity then need to be developed.

\section*{Acknowledgment} The authors would like to thank Joyce McLaughlin for fruitful discussions on Elastography and Yu Yuan for pointing out the reference \cite{M-SP-66}.  They are indebted to C\'edric Bellis and Fran\c cois Monard for their careful reading of the manuscript and their suggestions.   Part of this work was carried out during the program on Inverse Problems at the Newton Institute in 2011, and the authors would like to express their gratitude to the Newton Institute and the organizers of the program.
GB was partially funded by grants NSF DMS-1108608 and DMS-0804696. GU was partially funded by the NSF and a Rothschild Distinguished Visiting Fellowship at the Newton Institute. 

%%%%%%%%%%%%%%%%
%%%%%%%%%%%%%%%%%%
{\small

%\bibliography{../../bibliography} 
%\bibliographystyle{siam}
}
\end{document}